\documentclass{article}

\usepackage{amsmath,amsthm,amssymb}
\usepackage{enumitem}
\usepackage{mathtools}
\usepackage{mathrsfs}

\usepackage[a4paper, total={6in, 8in}]{geometry}

\usepackage[utf8]{inputenc}
\usepackage[english]{babel}

\numberwithin{equation}{section}

\newtheorem{theorem}{Theorem}[section]
\newtheorem{thmletter}{Theorem}

\newtheorem{proposition}[theorem]{Proposition}
\newtheorem{lemma}[theorem]{Lemma}
\newtheorem{corollary}[theorem]{Corollary}

\theoremstyle{definition}
\newtheorem{definition}[theorem]{Definition}
\newtheorem{remark}[theorem]{Remark}

\date{}

\newcommand\blfootnote[1]{%
	\begingroup
	\renewcommand\thefootnote{}\footnote{#1}%
	\addtocounter{footnote}{-1}%
	\endgroup
}

\def\N{\mathbb N}
\def\R{\mathbb R}
\def\C{\mathbb C}
\def\Z{\mathbb Z}

\title{Boas' problem on Hankel transforms}

\author{A. Debernardi\footnote{E-mail: \texttt{adebernardipinos@gmail.com}} \\ 
	\small{Centre de Recerca Matem\`atica, 08193, Bellaterra, Barcelona, Spain,}\\ \small{and Bar-Ilan University, 52900, Ramat-Gan, Israel}}

\begin{document}
	
\maketitle
\begin{flushleft}
	\small{AMS 2010 Primary subject classification: 42A38, 26D15. Secondary: 26A48.}
\end{flushleft}

\section*{Abstract}
Norm equivalences between a function and its Hankel transform are studied both in the context of weighted Lebesgue spaces with power weights, and in Lorentz spaces. Boas'-type results involving real-valued general monotone functions are obtained. Corresponding results for the Fourier transform are also given.

\blfootnote{\textbf{Acknowledgements}: The author acknowledges the remarks of the anonymous referee that contributed to significantly improve the manuscript.}

\blfootnote{This research was partially funded by the CERCA Programme of the Generalitat de Catalunya, Centre de Recerca Matem\`atica, the grant  MTM2017--87409--P from the Spanish Ministerio de Economía, Industria y Competitividad, the ERC starting grant No. 713927, and the ISF grant No. 447/16.}

\section{Introduction}
Given a $2\pi$-periodic function $f$ with Fourier series
\[
f(x)\sim \frac{a_0}{2}+\sum_{n=1}^\infty a_n \cos nx+b_n\sin nx,
\]
a classical problem is to study relations between the integrability of $f$  and the summability of its Fourier coefficients $\{a_n\}$, $\{b_n\}$. One of the most celebrated results in this direction is the Hardy-Littlewood theorem \cite{HaLi-Lp}, which states that for $1<q<\infty$ there exist a constant $C_q$ such that
\begin{equation}
\label{EQhl-charact}
C_q^{-1}\bigg(\int_0^{2\pi} |f(x)|^q\, dx\bigg)^{1/q}\leq \bigg(|a_0|^{q}+\sum_{n=1}^\infty n^{q-2}(|a_n|^{q}+|b_n|^{q})\bigg)^{1/q}\leq  C_q\bigg(\int_0^{2\pi} |f(x)|^q\, dx\bigg)^{1/q}.
\end{equation}

This representation of $L^q$ norms of functions via the weighted $\ell^q$ norms of their Fourier coefficients is useful for applications in other problems (cf. \cite{GT,NursultanovNetspaces,sagherQM} and the references therein). Thus, two interesting problems are to study what kind of weights may be incorporated in \eqref{EQhl-charact} and what generalizations of monotone sequences may be considered in such a way that relation \eqref{EQhl-charact} still holds. It is worth mentioning that such a relation with Lorentz norms instead of Lebesgue norms has also been object of study, although we omit such results for the sake of simplicity. 

Extensions of the equivalence \eqref{EQhl-charact} have been given for more general weights in \cite{sagherboas}, and the monotonicity condition was replaced by \textit{general monotonicity} in \cite{askhat,studia,TikGM,YZZ}, among several other works. It is worth to mention that a general monotone sequence need not be nonnegative (although it is a typical assumption in this kind of problem, in order to show the left-hand side inequality of \eqref{EQhl-charact}). Thus, one may wonder if \eqref{EQhl-charact} also holds when the nonnegativity assumption is replaced by a milder one. The answer is positive if we consider $\{a_n\}$ and $\{b_n\}$ to be real-valued and to satisfy the general monotonicity condition
\begin{equation}
\label{EQ-gmseq}
\sum_{k=n}^{2n}|c_k-c_{k+1}|\leq C\sum_{k=n/\lambda}^{\lambda n}\frac{|c_k|}{k}, \qquad c_k=a_k,\, b_k,
\end{equation}
for all $n$, where $C,\lambda>1$ are absolute constants. More precisely, in \cite{askhat}, the authors proved that for real-valued sequences $\{a_n\}$ and $\{b_n\}$ satisfying \eqref{EQ-gmseq}, the equivalence
\[
\bigg(|a_0|^q+\sum_{n=1}^\infty n^{q/p'-1}(|a_n|^q+|b_n|^q)\bigg)^{1/q}\asymp \bigg(\int_0^{2\pi} x^{q/p-1}|f(x)|^q\, dx\bigg)^{1/q}, \quad 1< p< \infty, \quad 1\leq q\leq \infty,
\]
holds, with the usual modification for $q=\infty$. Here and in what follows the symbol $\asymp$ is defined as follows: if  $A\leq CB$, where $C$ is independent of essential quantities of $A$ and $B$, we write $A \lesssim B$. Likewise, $A\gtrsim B$ will denote $A\geq CB$, and if $A\lesssim B$ and $A\gtrsim B$ simultaneously, we write $A \asymp B$.  

A converse equivalence for Lorentz norms was also obtained in \cite{askhat}, i.e.,
\[
\bigg((a_0^*)^q+\sum_{n=1}^\infty n^{q/p'-1}((a_n^*)^q+(b_n^*)^q)\bigg)^{1/q}\asymp \bigg(\int_0^{2\pi} x^{q/p-1}f^*(x)^q\, dx\bigg)^{1/q}, \quad 1< p< \infty, \quad 1\leq q\leq \infty,
\]
where $f^*$ denotes the decreasing rearrangement of $f$ (defined below), and $\{a^*_n\}$ and $\{b_n^*\}$ are the decreasing rearrangements of $\{a_n\}$ and $\{b_n\}$ respectively, or in other words, the sequences $\{|a_n|\}$ and $\{|b_n|\}$ rearranged in decreasing order. 

Before discussing the analogous inequalities to those presented above for Fourier transforms instead of Fourier series, let us introduce the basic notions we will use. All functions considered in this paper will be defined on an interval of $\R$ (mostly on $\R_+:=(0,\infty)$) and Lebesgue measurable. Non-weight functions are also assumed to be locally integrable on their interval of definition.

For $0<q\leq \infty$ and a weight $w:\R_+\to \R_+$, the weighted Lebesgue space $L^q(w)$ is defined as the set of all complex-valued measurable functions $f$ for which the functional 
\[
\Vert f\Vert_{L^q(w)}:=\begin{dcases}
\bigg( \int_0^\infty  w(x)|f(x)|^q\, dx\bigg)^{1/q},& \text{if }0<q<\infty,\\
\sup_{x\in (0,\infty)} w(x)|f(x)|, & \text{if } q=\infty,
\end{dcases}
\]
is finite. A particular example of weighted Lebesgue space that plays a significant role in this paper is the space $L^q(w)$ with $w(x)=x^{q/p-1}$ and $0<p\leq\infty$ (where in the case $p=\infty$ we take the convention $1/p\equiv 0$ and in the case $q=\infty$, we set $w(x)=x^{1/p}$). Following Sagher \cite{sagherboas}, such a space will be denoted by $L^q_{t(p,q)}$, and obviously $L^q_{t(q,q)}=L^q$. We may also refer to Lebesgue spaces for functions defined on $\R$; in this case the integration is obviously performed on $\R$, and if the corresponding functional is finite we write that $f\in L^q_\R(w)$.

We also define the Lorentz spaces $L^{p,q}$, introduced in \cite{Lorentzspaces} (see also \cite{BSbook}). To this end, recall that for a function $f$ defined on an interval $(a,b)\subset \R$, the distribution function of $f$ (with respect to the Lebesgue measure) is
\[
d_f(s)=|\{x\in(a,b):|f(x)|>s \}|,\qquad s\geq 0,
\]
where $|E|$ denotes the Lebesgue measure of a set $E$. The decreasing rearrangement of $f$ is then defined as
\[
f^*(x)=\inf \{s>0:d_f(s)\leq x\}, \qquad x\geq 0.
\]
The Lorentz space $L^{p,q}$, with $0<p,q\leq \infty$, is the set of all complex-valued functions $f$ defined on $(0,\infty)$ for which the functional
\[
\Vert f\Vert_{L^{p,q}}:=\begin{dcases}\bigg(\int_0^\infty \big(x^{1/p}f^*(x) \big)^{q}\frac{dx}{x}\bigg)^{1/q}, &\text{if } 0<p,q<\infty,\\
\sup_{x\in (0,\infty)} x^{1/p}f^*(x), &\text{if }0<p \leq \infty \text{ and }q=\infty,
\end{dcases}
\]
is finite. We will denote the corresponding Lorentz space of functions defined on $\R$ as $L^{p,q}_\R$. It is well known that for any $0<p\leq \infty$, $L^{p,p}=L^p$, and for any $0<q<r\leq \infty$, $L^{p,q}$ is a subspace of $L^{p,r}$ \cite{Grafakosbook}, or in other words, there exists a constant $C_{p,q,r}$ such that for every $f\in L^{p,q}$,
\[
\Vert f\Vert_{L^{p,r}}\leq C_{p,q,r}\Vert f\Vert_{L^{p,q}}.
\]
Note that if we restrict ourselves to considering only decreasing functions, the spaces $L^q_{t(p,q)}$ and $L^{p,q}$ coincide. Another useful expression for the Lorentz norm is \cite{Grafakosbook}
\begin{equation}
\label{EQlorentz-norm-alternative}
\Vert f\Vert_{L^{p,q}}=\begin{dcases}
p^{1/q}\bigg(\int_0^\infty \big(s d_f(s)^{1/p}\big)^q\frac{ds}{s}\bigg)^{1/q}, &\text{if }0<p,q<\infty,\\
\sup_{s\in(0,\infty)}sd_f(s)^{1/p},&\text{if }q=\infty.
\end{dcases}
\end{equation}

As one may expect, the equivalence \eqref{EQhl-charact} has its analog in the case of Fourier transforms, whose one-dimensional version reads as
\begin{equation}
\label{EQhali-functions}
\bigg(\int_{\R} |x|^{q-2} |\widehat{f}(x)|^q\, dx\bigg)^{1/q} \asymp \bigg(\int_{\R}|f(x)|^q\, dx\bigg)^{1/q}, \qquad 1<q<\infty,
\end{equation}
for any even function $f$ nonincreasing on $(0,\infty)$ (cf. \cite[Ch. IV]{titchmarsh}). What is more, Boas conjectured \cite{boas} that a similar relation to \eqref{EQhali-functions} with weights should be satisfied for sine and cosine transforms. More precisely, the conjecture is as follows. Let $G(x)=\int_0^\infty  g(t) \varphi(xt)\, dt$, where $\varphi(s)$ is either $\sin s$ or $\cos s$. If $g$ is nonnegative and nonincreasing and $-1/q'<\gamma<1/q$, then
\begin{equation*}
\label{EQboasconjecture}
x^{\gamma+1-2/q}G(x) \in L^q \qquad \text{if and only if} \qquad t^{-\gamma}g(t)\in L^q, \qquad 1<q<\infty.
\end{equation*}

An extended version of this conjecture was proved by Sagher in \cite{sagherboas}, where he also considered Lorentz spaces $L^{p,q}$ in place of the Lebesgue spaces $L^q$. Recent developments on general monotone functions (whose definition is analogous to \eqref{EQ-gmseq}, cf. Section~\ref{SSGM}) gave rise to further generalizations of Boas' conjecture, see \cite{LTnach,TikSzeged,TikGM}. In particular, Boas' problem was studied for the one-dimensional Fourier transform \cite{LTparis} (see also \cite{LTannali}), the multidimensional Fourier transform of radial functions \cite{GLT}, and for Hankel transforms \cite{DCGT}. In these works the involved functions were assumed to be nonnegative. 

A complex-valued function defined on $\R_+$ and locally of bounded variation $f$ is said to be \textit{general monotone} ($f\in GM$) \cite{LTnach} if there exist constants $C>0$ and $\lambda>1$ (depending on $f$) such that
\begin{equation}
\label{EQgmdefinition}
\int_{x}^{2x}|df(t)|\leq C\int_{x/\lambda}^{\lambda x} \frac{|f(t)|}{t}\, dt, \qquad \text{for all }x>0.
\end{equation}

Our main goal is to prove a version of Boas' conjecture for Hankel transforms of general monotone functions with the assumption $g\geq 0$ replaced by $g$ real-valued, from which all previous results, such as Hardy-Littlewood theorem, can be derived. We also give corresponding integrability theorems on Lorentz spaces. We emphasize that results in this direction were obtained  very recently  for Fourier series in the paper \cite{askhat}.

For $\alpha\geq-1/2$, the Hankel transform of a function $f\in  L^1(0,\infty)$ (see \cite[Ch. VIII]{titchmarsh} and \cite[Ch. IV]{SWfourier}) is defined as
\begin{equation}\label{EQhankeltr}
H_\alpha f(y)=\int_0^\infty f(x)\sqrt{xy}J_\alpha(xy)\, dx, \qquad y\in \R_+,
\end{equation}
where $J_\alpha$ denotes the Bessel function of order $\alpha$ (cf. Subsection~\ref{SSbessel}). It is well known that Hankel transforms describe the Fourier transforms of radial functions defined on $\R^n$. More precisely, if $f\in L^1(\R^n)$ and $f(x)=f_0(|x|)$, its Fourier transform is also radial, and moreover
\begin{equation}
\label{EQ-fourier-radial-functions}
|y|^{\frac{n-1}{2}} \widehat{f}(y)=|y|^{\frac{n-1}{2}} \int_{\R^n}f(x) e^{i(x,y)}\, dx=c_n  H_{\frac{n}{2}-1}\big[s^{\frac{n-1}{2}}f_0(s) \big](|y|),
\end{equation}
see \cite[Ch. IV, Theorem 3.3]{SWfourier}. 
Furthermore, since the Fourier transform in one dimension can be written as a sum of two Hankel transforms (see Subsection~\ref{SSbessel} below), obtaining Boas-type results for Hankel transforms  allows to derive the corresponding theorems for the Fourier transform.

In what follows we consider the Hankel transform of $f$ to be the pointwise limit
\begin{equation}
\label{EQlimit}
H_\alpha f(y)=\lim_{\substack{M\to 0\\N\to \infty}}\int_M^N f(x)\sqrt{xy}J_\alpha (xy)\, dx.
\end{equation}

Our main results read as follows.
\begin{theorem}
	\label{THMmainlebesgue}
	Let $f\in GM$ be real-valued. For $1\leq q\leq \infty$ and $\dfrac{1}{\alpha+3/2}<p<\infty$, one has
	\begin{equation*}
	\label{EQboaslebesgue}
	\Vert x^{1/p'-1/q} H_\alpha f\Vert_{L^q}\asymp \Vert x^{1/p-1/q} f\Vert_{L^q},
	\end{equation*}
	or in other words, $f\in L^q_{t(p,q)}$ if and only if $H_\alpha f\in L^q_{t(p',q)}$.
\end{theorem}

\begin{theorem}
	\label{THMmainlorentz}
	Let $f\in GM$ be real-valued. For $1<p<\infty$ and $1\leq q\leq \infty$, one has
	\begin{equation*}
	\label{EQboaslorentz}
	\Vert H_\alpha f\Vert_{L^{p',q}}\asymp \Vert f\Vert_{L^{p,q}},
	\end{equation*}
	or in other words, $f\in L^{p,q}$ if and only if $H_\alpha f\in L^{p',q}$.
\end{theorem}

Theorem~\ref{THMmainlebesgue} was proved for nonnegative $f$ and $1<q<\infty$ in \cite{DCGT} (see also \cite{GLT} for the case of Fourier transforms of radial functions, and the earlier \cite{LTparis,LTannali} for the sine and cosine transforms). 

It is worth mentioning that the inequality $\lesssim$ in Theorem~\ref{THMmainlebesgue} is a particular case of the well-known Pitt's inequality (see, e.g.,  \cite{BH,DC,GLTPitt,NT,Pitt}). Such kind of inequalities are often studied excluding the cases $q=1,\infty$.

With Theorem~\ref{THMmainlebesgue} in hand, we can easily derive the promised integrability results for the Fourier transform in one dimension, and also for Fourier transforms of radial functions in several dimensions. The corresponding Boas theorem for the Fourier transform in one dimension reads as follows (a version of this result was proved for nonnegative $GM$ functions in \cite{bootonlorentz}).
\begin{corollary}\label{COR-fourier}
	Let $f$ be a function defined on $\R$ and such that the even and odd parts of $f$,
	\[
	f_e(x)=\frac{f(x)+f(-x)}{2},\qquad f_o(x)=\frac{f(x)-f(-x)}{2},
	\]
	are real-valued $GM$ functions \textnormal{(}when restricted to $(0,\infty)$\textnormal{)}. Then, for $1<p,q<\infty$, $v(x)=|x|^{q/p-1}$, and $w(x)=|x|^{q/p'-1}$,
	\[
	f\in L^{q}_{\R}(v)\Leftrightarrow \widehat{f}\in L^{q}_\R(w)\Leftrightarrow f\in  L^{p,q}_{\R}\Leftrightarrow  \widehat{f}\in L^{p',q}_\R.
	\]
\end{corollary}
In higher dimensions, identity~\eqref{EQ-fourier-radial-functions} allows to characterize power weights for which Pitt's inequality (on $\R^n$) for radial $GM$ functions holds.

\begin{corollary}\label{COR-fourier-radial}
	Let $f$ be a real-valued radial function defined on $\R^n$, i.e., $f(x)=f_0(|x|)$, and such that $f_0\in GM$. Then
	\[
		\int_{\R^n}|x|^{-\beta q}|\widehat{f}(x)|^q\, dx	\asymp \int_{\R^n}|x|^{\gamma q}|f(x)|^q\, dx\asymp \int_{\R_+} t^{n-1+\gamma q}|f_0(t)|^q\, dt,
	\] 
	if and only if $\gamma=\beta+n-\dfrac{2n}{q}$ and  $\dfrac{n}{q}-\dfrac{n+1}{2}<\beta<\dfrac{n}{q}$.
\end{corollary}
Finally, we also give a generalization of Hardy-Littlewood theorem for the Fourier transform of real-valued radial functions \cite[Ch. IV]{titchmarsh}, which immediately follows from Corollary~\ref{COR-fourier-radial} with the appropriate choice of $\beta$ and $\gamma$.
\begin{corollary}
	Let $f(x)=f_0(|x|)$ be a real-valued radial function defined on $\R^n$, and such that $f_0\in GM$. Then
	\[
	\int_{\R^n} |\widehat{f}(x)|^q\, dx \asymp \int_{\R^n} |x|^{n(q-2)}|f(x)|^q\, dx
	\]
	if and only if $\dfrac{2n}{n+1}<q<\infty$, and
	\[
	\int_{\R^n} |x|^{n(q-2)}|\widehat{f}(x)|^q\, dx\asymp \int_{\R^n} |f(x)|^q\, dx
	\]
	if and only if $1<q<\dfrac{2n}{n-1}$.
\end{corollary}

The paper is structured in the following way. In Section~\ref{SECprelim} we introduce the preparatory material concerning Hankel transforms, which includes their definition in the distributional sense. Section~\ref{SSGM} is devoted to the discussion of general monotone functions. In particular, we prove Theorem~\ref{THMaverageoperator}, which relates weighted norm inequalities between a general monotone function and its maximal averaging operator, a central tool to carry out this work. Section~\ref{SECdefinitenessHa} is devoted to find sufficient conditions on a function $f$ so that its Hankel transform $H_\alpha f$ is well defined as an improper integral (where we also assume $f$ is general monotone), and as a distribution. Finally, in Section~\ref{SECmainresults}, we put everything together in order to prove our main results, namely Theorems~\ref{THMmainlebesgue}~and~\ref{THMmainlorentz}. The mentioned results for the Fourier  transforms (Corollaries~\ref{COR-fourier}~and~\ref{COR-fourier-radial}) are also proved.

\section{Preliminary concepts}\label{SECprelim}
\subsection{Bessel functions}\label{SSbessel}
For $\alpha\geq -1/2$, the Bessel function of order $\alpha$, $J_\alpha$, is defined as
\[
J_\alpha(x)=\sum_{k=1}^\infty \frac{(-1)^k}{k!\Gamma(k+\alpha+1)}\Big(\frac{x}{2}\Big)^{\alpha+2k}, \qquad x>0,
\]
and the series converges absolutely and uniformly on every compact interval. Let us now mention some useful properties of $J_\alpha$, which can be found in  \cite{EMOT}, together with alternative definitions and several additional properties. First of all, we have the upper estimate
\begin{equation}
\label{EQbesselestimate}
|J_\alpha(x)|\leq \begin{cases}
C_\alpha x^{\alpha},&\text{if }x\leq 1,\\
C_\alpha x^{-1/2}, &\text{if }x>1,
\end{cases}
\end{equation}
or equivalently, $|J_\alpha(x)|\leq C_\alpha\min\{x^{\alpha},x^{-1/2}\}$. For $\alpha=\pm 1/2$, one has
\begin{equation}
\label{EQbesselfns1/2}
J_{-1/2}(x)=\sqrt{\frac{2}{\pi}} \frac{\cos x}{\sqrt{x}}, \qquad J_{1/2}(x)=\sqrt{\frac{2}{\pi}} \frac{\sin x}{\sqrt{x}},
\end{equation}
so that the cosine and sine transforms of $f$ are equal (up to a constant) to $H_{-1/2}f$ and $H_{1/2}f$, respectively.

For $\alpha>-1/2$, let us denote by 
\[
K^\alpha_y(x)=\int_0^x t^{1/2}J_\alpha(ty)\,dt,
\]
so that 
\begin{equation}
\label{EQprimitive}
\dfrac{d}{dx} K^\alpha_y(x)=x^{1/2}J_{\alpha}(xy).
\end{equation}
Such a function is well defined, since $J_\alpha$ is continuous and $tJ_{\alpha}(ty)$ vanishes as $t\to 0$. For $\alpha=-1/2$, it follows from \eqref{EQbesselfns1/2} that $K_y^{-1/2}(x)= \sqrt{\dfrac{2}{\pi}} \dfrac{\sin xy}{y^{3/2}}$ satisfies \eqref{EQprimitive}.

It is shown in \cite{miohankel} (see also \cite{DCGT}) that 
\begin{equation}
\label{EQprimitiveest}
|K^\alpha_y(x)|\lesssim y^{-3/2},\qquad x,y>0.
\end{equation}
This estimate is particularly useful when integrating by parts.

\subsection{Distributional Hankel transforms}

Under the assumption $f\in L^1(0,\infty)$, the integral in \eqref{EQhankeltr} is absolutely and uniformly convergent on $\R_+$, and if $H_\alpha f\in L^1(0,\infty)$, the inversion formula
\begin{equation}
\label{EQinversion}
f(x)=\int_0^\infty H_\alpha f(y)\sqrt{xy}J_\alpha(xy)\, dy
\end{equation}
holds. Furthermore, if $f$ and $G$ are in $L^1(0,\infty)$, and $F$ and $g$ denote the direct and inverse Hankel transforms of order $\alpha$ of $f$ and $G$, respectively, Parseval's formula
\begin{equation}
\label{EQparseval}
\int_0^\infty f(x)g(x)\, dx=\int_0^\infty F(x)G(x)\, dx
\end{equation}
holds.

However, such integrability conditions for the above theory to work are rather restrictive. We can define the Hankel transform of functions from wider spaces in the distributional sense, analogously as done for the Fourier transform \cite{Grafakosbook}, based on Parseval's formula \eqref{EQparseval}. We follow the theory of Zemanian. In  \cite{hankeldistrib}, he constructed, for any $\alpha\geq -1/2$, topological linear spaces $\mathcal{H}_\alpha$ of test functions defined on $(0,\infty)$ for which the Hankel transform $H_\alpha$ is an automorphism. We now briefly present the basic elements of this theory that will be useful for our purpose. Before proceeding further, we refer to \cite{IL,LTrebels,mizohata}, where the reader may also find a distributional approach to the Fourier transform of radial functions.

\begin{definition}
	A complex-valued function $\varphi\in C^\infty (0,\infty)$ belongs to $\mathcal{H}_\alpha$ if for any nonnegative integers $m,n$,
	\begin{equation}
	\label{EQseminorms}
	\gamma^\alpha_{m,n}(\varphi)=\sup_{x\in (0,\infty)}|x^m (x^{-1}D)^n(x^{-\alpha-1/2}\varphi(x))|<\infty,
	\end{equation}
	where $D=d/dx$.
\end{definition}
The space $\mathcal{H}_\alpha$ is linear over $\C$, and its topology is the one given by the seminorms \eqref{EQseminorms}.  In \cite{hankeldistrib}, the author also proved the following.
\begin{lemma}
	Let $\alpha\geq -1/2$. Then the Hankel transform $H_\alpha$ is an isomorphism from $\mathcal{H}_\alpha$ onto itself.
\end{lemma}
For a fixed $\alpha\geq -1/2$, the space $\mathcal{H}_\alpha$ in the theory of the Hankel transform (of order $\alpha$) plays an analogous role as the Schwartz space $\mathscr{S}$ in the theory of the Fourier transform. For a more exhaustive treatment of the spaces $\mathcal{H}_\alpha$, see Section 2 of \cite{hankeldistrib}.

Let us denote $\R_+:= (0,\infty)$. By $\mathcal{D}_{\R_+}$ we denote the space of smooth functions supported on  $\R_+$,  with the topology that makes its dual $\mathcal{D}_{\R_+}'$ the space of Schwartz distributions on $\R_+$ (cf. \cite[Ch. III]{schwartzdistr} for further details). Under these definitions, it turns out that 
\begin{lemma}
	The space $\mathcal{D}_{\R_+}$ is a subspace of $\mathcal{H}_\alpha$ for any $\alpha\geq -1/2$.
\end{lemma}
It should also be mentioned that the space $\mathcal{D}_{\R_+}$ is not dense in $\mathcal{H}_\alpha$.

The analogue to the space of tempered Schwartz distributions $\mathscr{S}'$ is defined as follows. We denote by $\mathcal{H}_\alpha'$ the dual space of $\mathcal{H}_\alpha$, which is a linear space. By $\langle T,\varphi\rangle$, we denote the complex number that $T\in \mathcal{H}_\alpha'$ assigns to $\varphi\in \mathcal{H}_\alpha$.

The spaces $\mathcal{H}_\alpha'$ are equipped with the weak topology generated by the seminorms
\[
\eta_\varphi(T):=|\langle T,\varphi\rangle|, \qquad \varphi\in \mathcal{H}_\alpha\text{ arbitrary.}
\]
Moreover, for any $T\in \mathcal{H}_\alpha'$, there exist $r\in \N \cup\{0\}$ and $C>0$ such that for every $\varphi\in \mathcal{H}_\alpha$,
\[
|\langle T,\varphi\rangle|\leq C \max_{\substack{0\leq m\leq r\\ 0\leq n\leq r}}\gamma_{m,n}^\alpha(\varphi),
\]
which is proved in an analogous way as its counterpart for tempered distributions \cite{zemanianbook}.

Let us now define the Hankel transform of a distribution $T\in \mathcal{H}_\alpha$. It is defined similarly as the Fourier transform of a tempered Schwartz distribution, that is, via Parseval's theorem \eqref{EQparseval}. 
\begin{definition}
	The Hankel transform of order $\alpha\geq -1/2$ of $T\in \mathcal{H}_\alpha'$, $H_\alpha T$, is defined by the relation
	\begin{equation}
	\label{EQhankeldefdistribution}
	\langle T,H_\alpha \varphi\rangle=\langle  H_\alpha T,\varphi\rangle,\qquad \varphi\in \mathcal{H}_\alpha.
	\end{equation}
\end{definition}
Relation \eqref{EQhankeldefdistribution} determines a functional $H_\alpha T$ on $\mathcal{H}_\alpha$, and it can also be used to define the inverse transform $H_\alpha^{-1}$. 
\begin{theorem}\label{THMisomorphism}
	Let $\alpha\geq -1/2$. Then the Hankel transform $H_\alpha$ is an isomorphism from $\mathcal{H}_\alpha'$ onto itself.
\end{theorem}

The ordinary Hankel transform defined for functions $f\in L^1(0,\infty)$ is then a special case of the distributional Hankel transform \eqref{EQhankeldefdistribution}.

We emphasize that all the results presented in this section can be found with more detail in Sections 2--5 of \cite{hankeldistrib}.

\section{General monotone functions}\label{SSGM}
The concept of general monotonicity (already defined in \eqref{EQgmdefinition}) was first introduced by Tikhonov for sequences in  \cite{TikSzeged,TikGM} (see also \cite{LTnach} for a comprehensive survey on $GM$ functions and sequences). Note that without loss of generality, if $f\in GM$, we may take a different $GM$ constant $\lambda'=2^\nu>\lambda$ with $\nu \in \N$ in place of $\lambda$. For convenience, we will use this property repeatedly.

 We now list some properties of $GM$ functions that will be useful later. 

\begin{lemma}[\cite{LTnach}]\label{LEMgmproperties} Let $f\in GM$.
	\begin{enumerate}[label=\textnormal{(}\roman{*}\textnormal{)}]
		\item The function $x^\gamma f(x)$ is general monotone for any $\gamma\in \R$.
		\item For any $t>0$ and any $u\in [t,2t]$, $|f(u)|\lesssim \displaystyle\int_{t/\lambda}^{\lambda t}\frac{|f(x)|}{x}\, dx$.
		\item For any $t>0$ and any $\gamma\in\R$, $\displaystyle\int_t^\infty x^\gamma|df(x)|\lesssim \displaystyle\int_{t/\lambda}^\infty x^{\gamma-1}|f(x)|\, dx$.
		\item Let $\varepsilon>0$. If $f\in L^1(0,\varepsilon)$, then $xf(x)\to 0$ as $x\to 0$. If $f\in L^1(\varepsilon,\infty)$, then $xf(x)\to 0$ as $x\to \infty$.
	\end{enumerate}
\end{lemma}
\begin{remark}
	It is shown in \cite{miohankelgm} that if instead of $f\in L^1(\varepsilon,\infty)$, the function $f$ is real-valued and $\int_{\varepsilon}^{\infty} f(t)\, dt$ converges in the improper sense, then $xf(x)\to 0$ as $x\to \infty$.
\end{remark}

The following result due to B. Booton \cite{booton} relates the Lorentz and weighted Lebesgue norms of $GM$ functions. It was originally stated in more generality, but we present  a simplified version that is enough for our purpose.
\begin{thmletter}\label{THMbooton}
	Let $f\in GM$. For $1<p<\infty$ and $1\leq q\leq \infty$, or $p=q=\infty$, one has
	\[
	\Vert f\Vert_{L^q_{t(p,q)}}\asymp \Vert f\Vert_{L^{p,q}}.
	\]
\end{thmletter}

Define, for $g,\varphi:\R_+\to \C$,
\begin{equation}
\label{EQdefauxiliary}
\Phi_g(t)=\langle \varphi_t,g\rangle=\int_0^\infty \varphi_t(u)g(u)\, du,
\end{equation}
where $\varphi_t(u)=t^{-1}\varphi(u/t)$. We also denote
\begin{equation}
\label{EQdefmaximal}
M\Phi_g(t)=\sup_{x\geq t}|\Phi_g(x)|.
\end{equation}
Note that if $\varphi=\chi_{(0,1)}$, then $M\Phi_g(t)=\displaystyle \sup_{x\geq t}\bigg| \frac{1}{x}\int_0^x g(u)\, du\bigg|$. We now aim to prove a norm inequality for a weighted averaging operator applied to $GM$ functions, which is the key result of our approach. The statement is as follows.

\begin{theorem}\label{THMaverageoperator}
	Let $0<q\leq \infty$. Let $g\in GM$ be real valued, vanishing at infinity, and such that $x^r g(x)\to 0$ as $x\to 0$ for some $r>0$.  Define $\varepsilon = \dfrac{1}{C^4 2^{6r\nu+8\nu+16}}$, assume $\varphi:\R_+\to [0,1]$ is supported on the interval $(0,1+\varepsilon/2)$, and that $\varphi(x)\equiv 1$ for $x\in (0,1)$. Let $w:\R_+\to \R_+$ be a weight satisfying $w(s)\asymp w(t)$ for all $s,t\in [x,2x]$ and $x>0$. Then
	\[
	\Vert g\Vert_{L^q(w)}\lesssim \Vert M\Phi_g \Vert_{L^q(w)}.
	\]	
\end{theorem}

In order to prove Theorem~\ref{THMaverageoperator} we need some auxiliary results. From now on, we assume without loss of generality that the $GM$ constant $\lambda$ (see \eqref{EQgmdefinition}) equals $2^\nu$ for some $\nu\in \N$. Let us define, for any function $g\in GM$ and any  $n\in \Z$,
\begin{align*}
A_n&:=\sup_{2^n\leq t\leq 2^{n+1}}|g(t)|,\\
B_n&:=\sup_{2^{n-2\nu}\leq t\leq 2^{n+2\nu}}|g(t)|.
\end{align*}

Given $r>0$, for $n\in \Z$, we say that $n$ is a \textit{good} number if $B_{n}\leq 2^{2r\nu}A_{n}$. The rest of integer numbers consists of \textit{bad} numbers. Recall that the constant $\nu$ comes from the $GM$ condition. The parameter $r$ will be arbitrarily chosen at each point according to our convenience. In contrast with \cite{miohankelgm,DTgmlipschitz}, here we consider a slightly different definition of good numbers by incorporating the parameter $r>0$ (in the cited papers $r=2$ is fixed). The reason to do this is that every power function $x^\rho$ (which is a $GM$ function for any $\rho$) will have an infinite amount of good numbers if $r$ is chosen appropriately according to $\rho$. We give two examples illustrating this fact. On the one hand, if $g(x)=x^{-2}$, since
$$
A_n=\frac{1}{2^{2n}}, \qquad B_n= \frac{1}{2^{2n -4\nu}},
$$
then $B_n=2^{4\nu}A_n$, and all natural numbers $n$ (associated to $g$) are good (with $r=2$). On the other hand, if $g(x)=x^{-3}$, since
$$
A_n=\frac{1}{2^{3n}}, \qquad B_n= \frac{1}{2^{3n-6\nu}},
$$
then $B_n=2^{6\nu}A_n$, thus all natural numbers $n$ are good if $r=3$, and bad if $r=2$.
\begin{lemma}
	\label{LEMgood1-functions}
	Let $g$ be a $GM$ function. For any good number $n$, there holds
	\begin{equation}
	\label{EQgoodlemma1-functions}
	|E_n|:=\bigg| \bigg\{ x\in[2^{n-\nu},2^{n+\nu}]: |g(x)|>\frac{A_n}{C2^{2\nu+3}}\bigg\}\bigg|\geq   \frac{2^n}{C 2^{2r\nu+\nu+3}},
	\end{equation}
	where $C$ and $\nu$ are the constants from the $GM$ condition.
\end{lemma}
\begin{proof}
The proof just consists on rewriting that of \cite{miohankelgm} in the context of functions, with the difference that in the mentioned work the parameter $r=2$ is fixed (see also \cite{DTgmlipschitz}, where this idea was originally carried out for sequences). Assume \eqref{EQgoodlemma1-functions}	does not hold for $n\in \Z$. Let us define $D_n:=[2^{n-\nu},2^{n+\nu}]\backslash E_n$. Then, since $n$ is good,
		\begin{align*}
		\int_{2^{n-\nu}}^{2^{n+\nu}} \frac{|g(x)|}{x}\, dx& =\int_{D_n} \frac{|g(x)|}{x}\, dx + \int_{E_n} \frac{|g(x)|}{x}\, dx\\
		&\leq \frac{2^{n+\nu}A_n}{C2^{n-\nu}2^{2\nu+3}}+\frac{2^n B_n}{C2^{n-\nu}2^{2r\nu+\nu+3}}= \frac{A_n}{8C}+\frac{B_n}{C 2^{2r\nu+3}}\leq \frac{A_n}{4C}.
		\end{align*}
		The $GM$ condition implies that for any $x\in [2^n,2^{n+1}]$, 
		$$
		|g(x)|\geq A_n-\int_{2^n}^{2^{n+1}}|dg(t)|\geq A_n -C\int_{2^{n-\nu}}^{2^{n+\nu}} \frac{|g(t)|}{t}\, dt \geq A_n - \frac{A_n}{4}>\frac{A_n}{2},
		$$
		which contradicts our assumption.
\end{proof}
Note that in particular, Lemma~\ref{LEMgood1-functions} implies that if $n$ is a good number, then $A_n>0$. Before stating the next lemma, let us introduce the following notation:
$$
E_n^+:=\{ x\in E_n : g(x)>0\}, \qquad E_n^-:=\{ x\in E_n: g(x)\leq 0\}.
$$
\begin{lemma}
	\label{LEMgood2-functions}
	Let $g$ be a real-valued $GM$ function. For any good number $n$ there is an interval $(\ell_n,m_n)\subset [2^{n-\nu},2^{n+\nu}]$ such that at least one of the following holds\textnormal{:}
	\begin{enumerate}
		\item for any $x\in (\ell_n,m_n)$, there holds $g(x)\geq 0$ and
		$$
		|E_n^+ \cap (\ell_n,m_n)|\geq \frac{2^n}{C^3 2^{4r\nu+5\nu+12}};
		$$
		\item for any  $x\in (\ell_n,m_n)$, there holds $g(x)\leq 0$ and
		$$
		|E_n^- \cap (\ell_n,m_n)|\geq \frac{2^n}{C^3 2^{4r\nu+5\nu+12}},
		$$
	\end{enumerate}
	where $C$ and $\nu$ are the constants from the $GM$ condition and $r$ is the parameter from the definition of good numbers.
\end{lemma}
	\begin{proof}
		On the first place, by Lemma~\ref{LEMgood1-functions}, either $|E_n^+|\geq \dfrac{2^n}{C2^{2r\nu+\nu+4}}$ or $|E_n^-|\geq \dfrac{2^n}{C2^{2r\nu+\nu+4}}$. We assume the former, and prove that item 1. holds.
		
		Let us construct a system of disjoint intervals $\{I_j=[s_j,t_j]\}_{j=1}^{p_n}$ in $\big[ 2^{n-\nu},2^{n+\nu}+\varepsilon 2^n\big]$ (where $\varepsilon<1$ will be conveniently chosen later) as follows: Let $s_1=\inf E_n^+$, and
		$$
		\tau_1=\inf\{x\in [s_1,2^{n+\nu}]: g(x)\leq 0\}.
		$$
		If such $\tau_1$ does not exist, then we simply let $t_1=2^{n+\nu}$ and the conclusion follows with $(\ell_n,m_n)=(s_1,t_1)$.	Contrarily, we define
		$$
		t_1=\tau_1+\varepsilon 2^n.
		$$
		Once we have the first interval $I_1=[s_1,t_1]$, if $|E_n^+\backslash I_1|>0$, we let $s_2=\inf E_n^+\backslash I_1$, and define $\tau_2$ similarly as above, thus obtaining a new interval $I_2=[s_2,t_2]$. We continue this process until our collection of intervals is such that
		$$
		|E_n^+\backslash (I_1\cup I_2\cup \cdots \cup I_{p_n})|=0.
		$$ 
				
		By construction, for any $1\leq j\leq p_n-1$, we can find $y_j\in [s_j,\tau_j]$ such that $y_j\in E_n^+$, and $z_j\in [\tau_j,t_j]$ such that $g(z_j)\leq 0$. Thus,
		$$
		\int_{I_j}|dg(t)|=\int_{s_j}^{t_j}|dg(t)|\geq g(y_j)-g(z_j)\geq g(y_j)>\frac{A_n}{C2^{2\nu+3}}.
		$$
		Hence,
		$$
		\int_{2^{n-\nu}}^{2^{n+\nu}}|dg(t)|\geq \sum_{j=1}^{p_n-1}\int_{I_j}|dg(t)|\geq (p_n-1)\frac{A_n}{C2^{2\nu+3}}.
		$$
		On the other hand, the $GM$ property and the fact that $n$ is good imply that
		\begin{align*}
		\int_{2^{n-\nu}}^{2^{n+\nu}} |dg(t)|&\leq C2\nu \int_{2^{n-2\nu}}^{2^{n+2\nu}} \frac{|g(x)|}{x}\, dx \leq C2\nu B_n\int_{2^{n-2\nu}}^{2^{n+2\nu}} \frac{1}{x}\, dx\\
		&= C2\nu B_n \log 2^{4\nu} \leq C2^{2r\nu}8\nu^2 A_n\log 2  \leq C 2^{2r\nu+2\nu+3}A_n.
		\end{align*}		
		We can deduce from the above estimates that
		$$
		p_n \leq C^2 2^{2r\nu+4\nu+6} +1 \leq C^2 2^{2r\nu+4\nu+7}.
		$$
		By the pigeonhole principle (or Dirichlet's box principle), there is an integer $j$ such that 
		$$
		|E_n^+ \cap I_j|\geq \frac{2^n}{C^3 2^{4r\nu+5\nu+11}}.
		$$
		Given this $j$, we set $\varepsilon=\dfrac{1}{C^3 2^{4r\nu+5\nu+12}}$ and $(\ell_n,m_n)=(s_j,t_j-\varepsilon 2^n)=(s_j,\tau_j)\subset [2^{n-\nu},2^{n+\nu}]$, and the result follows.
	\end{proof}

Concerning bad numbers, we have the following result.

\begin{lemma}\label{LEMseqbadnrs}
	Let $g\in GM$ be vanishing at infinity and such that $x^{r_0}g(x)\to 0$ as $x\to 0$ for some $0\leq r_0\leq r$. Then, for every bad number $m\in \Z$ there exists either a finite sequence
	\begin{equation}
	\label{EQseq-}
	m=\gamma_0>\gamma_1>\cdots >\gamma_s:=\gamma_{m,s},
	\end{equation}
	or
	\begin{equation}
	\label{EQseq+}
	m=\gamma_0<\gamma_1<\cdots <\gamma_s:=\gamma_{m,s},
	\end{equation}
	such that $\gamma_0,\gamma_1,\ldots ,\gamma_{s-1}$ are bad, $\gamma_s$ is good, and the inequalities
	\[
	 A_{\gamma_j}<2^{-2r\nu}A_{\gamma_{j+1}},  \qquad |\gamma_j-\gamma_{j+1}|\leq 2\nu,
	\]
	hold for every $0\leq j\leq s-1$. In particular, there are infinitely many good numbers associated to $g$.
\end{lemma}
\begin{proof}
	Let $m\in \Z$ be a bad number. Then $A_m<2^{-2r\nu}B_m$, and we can find $\gamma\in \Z$ satisfying $B_m=A_\gamma$ and $|m-\gamma|\leq 2\nu$. Let
	\[
	\gamma_1=\min\{\gamma\in [m-2\nu,m+2\nu)\cap \Z:A_\gamma=B_m\}.
	\]
	We now have two possibilities, either $\gamma_1<m$, or $\gamma_1>m$. Assume first $\gamma_1<m$. Then either $\gamma_1$ is a good number, or there exists $\gamma$ satisfying $|\gamma_1-\gamma|\leq 2\nu$ for which $B_{\gamma_1}=A_\gamma$. Note that in this case $\gamma<\gamma_1$, otherwise we arrive at a contradiction. Set
	\[
	\gamma_2=\min\{\gamma\in [\gamma_1-2\nu,\gamma_1)\cap\Z: A_\gamma=B_{\gamma_1}    \}.
	\]
	Continuing this procedure, we can prove that we eventually find a good number $\gamma_s$, so that the sequence
	\[
	m=\gamma_0>\gamma_1>\cdots >\gamma_s
	\]
	is such that $\gamma_0,\ldots ,\gamma_{s-1}$ are bad numbers, and $\gamma_j-\gamma_{j+1}\leq 2\nu$ for $0\leq j\leq s-1$. Indeed, if we could not find such a $\gamma_s$, then there would exist an infinite sequence of bad numbers
	\[
	\gamma_0>\gamma_1>\cdots >\gamma_s>\gamma_{s+1}>\cdots,
	\]
	so that $\gamma_{j-1}-\gamma_j\leq 2\nu$ and
	\begin{equation}
	\label{EQbadnrs}
	0<2^{-2r\nu}A_{\gamma_1}<2^{-4r\nu}A_{\gamma_2}<\cdots <2^{-2rj\nu}A_{\gamma_j},
	\end{equation}
	for all $j\geq 1$. Now, note that
	\begin{align*}
	r_0\gamma_j&\geq r_0\gamma_0-2r_0j\nu. 
	\end{align*}
	Combining this with \eqref{EQbadnrs}, we obtain, since $0\leq r_0\leq r$,
	\[
	2^{r_0\gamma_j}A_{\gamma_j}>2^{r_0\gamma_0-2r_0j\nu+2r\nu(j-1)}A_{\gamma_1}=2^{r_0\gamma_0-2r\nu}A_{\gamma_1}2^{2j\nu(r-r_0)}\geq 2^{r\gamma_0-2r\nu}A_{\gamma_1}>0.
	\]
	Letting $j\to \infty$, we find that $x^{r_0}g(x)\not \to 0$ as $x\to 0$, which contradicts our hypotheses. This concludes the part of the proof corresponding to the case $\gamma_1<m$. Let us now assume $\gamma_1>m$. Then either $\gamma_1$ is good, or there exists $\gamma>\gamma_1$ such that $\gamma-\gamma_1\leq 2\nu-1$, and $A_{\gamma_1}<2^{-2r\nu}B_{\gamma_1}=2^{-2r\nu}A_{\gamma}$ (the case $\gamma<\gamma_1$ is not possible, as it leads to a contradiction). We now define
	\[
	\gamma_2=\max\{ \gamma\in (\gamma_1,\gamma_1+2\nu)\cap \Z: A_\gamma=B_{\gamma_1}\},
	\]
	and similarly as above, we can continue this procedure and obtain a finite sequence
	\begin{equation}
	\label{EQbadnrs2}
	m=\gamma_0<\gamma_1<\cdots<\gamma_s,
	\end{equation}
	where the numbers $\gamma_0,\ldots , \gamma_{s-1}$ are bad, $\gamma_s$ is good, and moreover,
	\[
	\gamma_{j+1}-\gamma_j\leq 2\nu-1, \qquad A_{\gamma_j}<2^{-2r\nu}A_{\gamma_{j+1}},
	\]
	for all $0\leq j\leq s-1$. If we could not find the finite sequence from \eqref{EQbadnrs2}, then there would exist an infinite sequence of bad numbers
	\[
	\gamma_0<\gamma_1<\cdots<\gamma_s<\gamma_{s+1}<\cdots,
	\]
	and we would obtain
	\[
	A_{\gamma_{s_2}}>2^{2r(s_2-s_1)\nu}A_{\gamma_{s_1}}, \qquad  s_1,s_2\geq 0,
	\]
	thus contradicting the hypothesis that $g$ vanishes at infinity.
\end{proof}

Note that in the proof Lemma~\ref{LEMseqbadnrs}, for any bad number $m$, the number $\gamma_{m,s}$ obtained in \eqref{EQseq-} or \eqref{EQseq+} is uniquely determined. The natural number $s$ will be called the \textit{length} of the bad number $m$. We also define the sets
\begin{align*}
Q_n^1&:=\{m\in \Z: m\text{ is a bad number and } \eqref{EQseq-}\text{ holds with }\gamma_{m,s}=n\},\\
Q_n^2&:=\{m\in \Z: m\text{ is a bad number and } \eqref{EQseq+}\text{ holds with }\gamma_{m,s}=n\},
\end{align*}
and note that $Q_{n_1}^j\cap Q_{n_2}^k=\emptyset$ for every $n_1,n_2\in \Z$ and $j,k=1,2$. Moreover, if we denote by $G$ the set of good numbers, one has
\[
\Z=G\cup\bigg(\bigcup_{n\in G}Q_n^1\bigg)\cup\bigg(\bigcup_{n\in G}Q_n^2\bigg),
\]
where all the unions are disjoint.

\begin{remark}\label{REMseqgoodnrs}
	For any good number $n$ and any $s\in \N$, each of the sets $Q_n^j$, $j=1,2$, contain at most $(2\nu)^s$ bad numbers of length $s$. Let us discuss the case $j=1$ (the case $j=2$ is analogous). Indeed, if $m\in Q_n^1$ is a bad number  such that the construction \eqref{EQbadnrs} yields $\gamma_{m,1}=n$, then necessarily $m\in [n-2\nu,n)\cap \Z$, so that there are at most  $2\nu$ bad numbers of length $1$ in $Q_n^1$. If $m$ is a bad number such that the construction \eqref{EQbadnrs} yields $\gamma_{m,2}=n$, we should count all possible choices of $\gamma_1,\gamma_2$ satisfying
	\begin{equation}
	\label{EQauxbadnrs}
	m>\gamma_1>\gamma_2=\gamma_{m,2}.
	\end{equation}
	We know that there are at most $2\nu$ possible choices of $\gamma_1$, and that $\gamma_2\in [\gamma_1-2\nu,\gamma_1)\cap\Z$, so that there are at most $(2\nu)^2$ possible choices $\gamma_1,\gamma_2$ satisfying \eqref{EQauxbadnrs}. Continuing the argument inductively proves our claim. 
\end{remark}

We are now in a position to prove Theorem~\ref{THMaverageoperator}.

\begin{proof}[Proof of Theorem~\ref{THMaverageoperator}]
	We first prove that for any good number $n$, the inequality $M\Phi_g(2^{n-\nu})\gtrsim A_n$ holds. Indeed, let $(\ell_n,m_n)\subset [2^{n-\nu},2^{n+\nu}]$ be the interval obtained from Lemma~\ref{LEMgood2-functions}. Then
	\begin{align*}
	M\Phi_g(2^{n-\nu})&=\sup_{x\geq 2^{n-\nu}}\bigg| \frac{1}{x}\int_0^x g(u)\, du + \frac{1}{x} \int_x^{x(1+\varepsilon)} \varphi\Big(\frac{u}{x}\Big)g(u)\, du \bigg|\\
	&\geq \frac{1}{2}\bigg|\frac{1}{m_n}\int_0^{\ell_n} g(u)\, du+ \frac{1}{m_n}\int_{\ell_n}^{\ell_n(1+\varepsilon)}\varphi\Big(\frac{u}{\ell_n}\Big)g(u)\, du\bigg|\\
	&\phantom{=}+ \frac{1}{2}\bigg|\frac{1}{m_n}\int_0^{m_n} g(u)\, du+ \frac{1}{m_n}\int_{m_n}^{m_n(1+\varepsilon)}\varphi\Big(\frac{u}{m_n}\Big)g(u)\, du\bigg|\\
	&\geq \frac{1}{2 m_n}\bigg|\int_{\ell_n}^{m_n} g(u)\, du +\int_{m_n}^{m_n(1+\varepsilon)}\varphi\Big(\frac{u}{m_n}\Big) g(u)\, du- \int_{\ell_n}^{\ell_n(1+\varepsilon)}\varphi\Big(\frac{u}{\ell_n}\Big)g(u)\, du\bigg|\\
	&\geq \frac{1}{2 m_n}\bigg| \int_{\ell_n}^{m_n} g(u)\, du\bigg|- \frac{1}{2m_n}\bigg( \int_{\ell_n}^{\ell_n(1+\varepsilon)} |g(u)|\, du+ \int_{m_n}^{m_n(1+\varepsilon)} |g(u)|\, du\bigg)\\
	&\geq \frac{1}{2 m_n}\bigg| \int_{\ell_n}^{m_n} g(u)\, du\bigg| - \frac{\varepsilon}{2} B_n \geq \frac{1}{2 m_n}\bigg| \int_{\ell_n}^{m_n} g(u)\, du\bigg|  -2^{2r\nu-1}\varepsilon A_n.
	\end{align*}
	By Lemma~\ref{LEMgood2-functions}, we have
	\[
	\frac{1}{2 m_n}\bigg| \int_{\ell_n}^{m_n} g(u)\, du\bigg|> \frac{A_n}{C^4 2^{4r\nu+8\nu+16}},
	\]
	and thus, by the choice of $\varepsilon$, we obtain
	\begin{equation}
	\label{EQestgoodnur}
	M\Phi_g(2^{n-\nu})> \frac{A_n}{C^4 2^{4r\nu+8\nu+16}}-2^{2r\nu-1}\varepsilon A_n=\frac{A_n}{C^4 2^{4r\nu+8\nu+17}},
	\end{equation}
	valid for any good number $n$. Let us now consider two subcases, namely $0<q<\infty$ and $q=\infty$. Let $G\subset \Z$ denote the set of good numbers associated to $g$ and $W(x)=\sup_{t\in [x,2x]}w(t)$. Then, for the case $0<q<\infty$,
	\begin{align*}
	\Vert g\Vert_{L^q(w)}^q&=\int_0^\infty w(t)|g(t)|^q\,dt= \sum_{n\in \Z}\int_{2^n}^{2^{n+1}}w(t)|g(t)|^q\,dt  \\
	&\lesssim \sum_{n\in \Z}W(2^n)A_n^q=\sum_{n\in G}W(2^n)A_n^q+\sum_{n\in G}\sum_{m\in Q_n^1\cup Q_n^2}W(2^n)A_n^q=:S_1+S_2.
	\end{align*}
	On the one hand, by \eqref{EQestgoodnur}
	\[
	S_1\lesssim \sum_{n\in G}W(2^n)M\Phi_g(2^{n-\nu})^q\lesssim \Vert M\Phi_g\Vert_{L^q(w)}^q,
	\]
	where the last inequality follows from the fact that $M\Phi_g$ is nonincreasing. On the other hand, in order to estimate $S_2$, let us first observe that there exists $B>0$ such that for every $k,m\in \N$,
	\[
	W(2^{m})\leq B\cdot W(2^{m\pm 1})\leq \cdots \leq B^{k}W(2^{m\pm k}).
	\] 
	Now, for any bad number $m\in Q_n^1$ of length $s$, it follows from the inequalities $m\leq n+2\nu s$ and $A_m<2^{-2rs\nu}A_n$ (cf. Lemma~\ref{LEMseqbadnrs}) that
	\begin{align*}
	W(2^m)A_m^q < W(2^m) A_n^q  2^{-2rs\nu q} \leq B^{2\nu s}2^{-2rs\nu q} W(2^n)A_n^q =2^{2\nu s(\log_2 B-rq)} W(2^n)A_n^q,
	\end{align*}
	and similarly, for any bad number $m\in Q_n^2$, it follows from the inequalities $n \leq m+2\nu s$ and $A_m<2^{-2rs\nu}A_n$ that
	\[
	W(2^m)A_m^q <2^{2\nu s(\log_2 B-rq)} W(2^n)A_n^q.
	\]
	From now on, we now assume without loss of generality that $r>q^{-1}\log_2 B$. By  Remark~\ref{REMseqgoodnrs}
	\begin{align*}
	S_2&= \sum_{n\in G}\sum_{m\in Q_n^1\cup Q_n^2} W(2^m)A_m^q < 2\sum_{n\in G}W(2^n)A_n^q \sum_{s=1}^\infty (2\nu)^s 2^{2\nu s(\log_2 B-rq)} \\
	&\lesssim \sum_{n\in G}W(2^n)A_n^q\lesssim \Vert M\Phi_g\Vert_{L^q(w)}^q,
	\end{align*}
	which concludes the proof of the case $0<q<\infty$. For the case $q=\infty$, the proof is similar. First of all, note that
	\[
	\sup_{\substack{n\in \Z\\ n\in G}}W(2^n)A_n\lesssim \sup_{\substack{n\in \Z\\ n\in G}}W(2^n) M\Phi_g(2^{n-\nu})\asymp  \Vert M\Phi_g\Vert_{L^{\infty}(w)}.
	\]
	Further, for any bad number $m\in Q_n^1$ of length $s$, it follows from the inequalities $m\leq n+2\nu s$ and $A_m<2^{-2rs\nu}A_n$ that
	\[
	W(2^m)A_m < W(2^m) A_n  2^{-2rs\nu q} \leq 2^{2\nu s(\log_2 B-rq)} W(2^n)A_n \lesssim  \Vert M\Phi_g\Vert_{L^{\infty}(w)}.
	\]
	Finally, if $m\in Q_n^2$ has length $s$, it follows from the inequalities $n\leq m+2s\nu$ and $A_m<2^{-2rs\nu}A_n$ that
	\[
	W(2^m)A_m<2^{2\nu s(\log_2 B-rq)} W(2^n)A_n \leq \Vert M\Phi_g\Vert_{L^{\infty}(w)}.
	\]
	Joining the above estimates we get $\Vert g\Vert_{L^{\infty}(w)} \lesssim \Vert M\Phi_g\Vert_{L^{\infty}(w)}$.
\end{proof}

\section{Well-definiteness of $H_\alpha f$ in function spaces}\label{SECdefinitenessHa}
In this section we show that the Hankel transform $H_\alpha f$ is well defined both as the pointwise limit \eqref{EQlimit} (provided that $f$ is $GM$) and also as an element of $\mathcal{H}_\alpha'$, whenever $f$ is from a suitable function space. Both facts put together imply that the inversion formula \eqref{EQinversion} holds almost everywhere for general monotone functions from such a space, in virtue of Theorem~\ref{THMisomorphism}.

\subsection{Pointwise convergence of $H_\alpha f$}

The goal is to show that the limit \eqref{EQlimit} exists for all $y\in \R_+$ whenever $f$ is from certain function spaces; in other words, $H_\alpha f$ is well defined as an improper integral.
\begin{lemma}\label{LEMimproperdef}
	Let $f\in GM$ and $1\leq q\leq \infty$.
	\begin{enumerate}
		\item If $f\in L^q_{t(p,q)}$ with $\dfrac{1}{\alpha+3/2}<p<\infty$, or
		\item if $f\in L^{p,q}$ with $1<p<\infty$, 
	\end{enumerate}
	then the limit
	\[
	\lim_{\substack{M\to 0\\N\to \infty}}\int_M^N f(x)\sqrt{xy}J_\alpha (xy)\, dx
	\]
	exists for all $y\in \R_+$.
\end{lemma}
\begin{proof}
	We show that for $f\in L^q_{t(p,q)}$ with $p,q$ as in the hypotheses and given $y\in \R_+$,
	\[
	\lim_{M\to 0}\int_0^M |f(x)\sqrt{xy}J_\alpha (xy)|\, dx=0,\qquad \lim_{N_1,N_2\to \infty}\int_{N_1}^{N_2}  f(x)\sqrt{xy}J_\alpha (xy)\, dx=0.
	\]
	The result for $f\in L^{p,q}$ will follow just by Theorem~\ref{THMbooton}. Since $J_\alpha(z)\lesssim z^{\alpha}$ for all $z>0$, by H\"older's inequality, if $1<q<\infty$,
	\begin{align*}
	\int_0^M |f(x)\sqrt{xy}J_\alpha (xy)|\, dx&\lesssim y^{\alpha+1/2} \int_0^M |f(x)|x^{\alpha+1/2}\, dx \\
	&\lesssim \Vert f\Vert_{L^q_{t(p,q)}}\bigg(\int_0^M x^{(\alpha+1/2-1/p+1/q)q'}dx\bigg)^{1/q'}\to 0 \qquad \text{as }M\to 0,
	\end{align*}
	for any $\dfrac{1}{\alpha+3/2}<p<\infty$. If $q=1$, we have 
	\[
	y^{\alpha+1/2} \int_0^M |f(x)|x^{\alpha+1/2}\, dx \lesssim  M^{\alpha+3/2-1/p}\Vert f\Vert_{L^1_{t(p,1)}}\to 0 \qquad \text{as } M\to 0,
	\]
	 and  if $q=\infty$,
	\[
	y^{\alpha+1/2} \int_0^M |f(x)|x^{\alpha+1/2}\, dx \lesssim \Vert f\Vert_{L^\infty_{t(p,\infty)}}\int_0^M x^{\alpha+1/2-1/p}\, dx\to 0 \qquad \text{as } M\to 0,
	\]
	with all the estimates valid for any $\dfrac{1}{\alpha+3/2}< p<\infty$.
	
	Integrating by parts and using \eqref{EQprimitiveest}, we get
	\begin{align*}
	\bigg|\int_{N_1}^{N_2}f(x)\sqrt{xy}J_\alpha (xy)\, dx\bigg|&\leq \sqrt{y}\big( |f(N_2)K_y^\alpha(N_2)|+|f(N_1)K_y^\alpha(N_1)| \big)+\sqrt{y}\int_{N_1}^{N_2}|K_y^\alpha(x)df(x)|\\
	&\lesssim |f(N_1)|+|f(N_2)|+\int_{N_1}^{N_2}|df(x)|.
	\end{align*}
	By Lemma~\ref{LEMgmproperties}, for $1\leq q<\infty$, $x^{q/p}f(x)\to 0$ as $x\to \infty$, and so does $f(x)$ (in the case $q=\infty$, $f$ trivially vanishes at infinity). Therefore, 
	\[
	|f(N_1)|+|f(N_2)|+\int_{N_1}^{N_2}|df(x)|\lesssim \int_{N_1}^\infty |df(x)|\lesssim \int_{N_1/\lambda}^\infty \frac{|f(x)|}{x}\, dx,
	\]
	and by H\"older's inequality, for $1<q<\infty$,
	\[
	\int_{N_1/\lambda}^\infty \frac{|f(x)|}{x}\, dx \lesssim \Vert f\Vert_{L^q_{t(p,q)}} \bigg(\int_{N_1}^\infty x^{-1-q'/p}dx\bigg)^{1/q'}\to 0 \qquad \text{as }N_1\to \infty.
	\] 
	For $q=1$, it is clear that
	\[
	\int_{N_1/\lambda}^\infty \frac{|f(x)|}{x}\, dx \lesssim N_1^{-1/p}  \Vert f\Vert_{L^1_{t(p,1)}}\to 0 \qquad \text{as }N_1\to \infty,
	\]
	and finally, for $q=\infty$,
	\[
	\int_{N_1/\lambda}^\infty \frac{|f(x)|}{x}\, dx \lesssim \Vert f\Vert_{L^\infty_{t(p,\infty)}} \int_{N_1/\lambda}^\infty \frac{1}{x^{1+1/p}}\, dx \to 0 \qquad \text{as }N_1\to \infty.\qedhere
	\]
\end{proof}

\subsection{Weighted Lebesgue spaces $L^q(w)$}
We first give sufficient conditions on the weight $w$ so that $H_\alpha f\in \mathcal{H}'_\alpha$ whenever $f\in L^q(w)$.
\begin{proposition}\label{PROPdistrweightedlebesgue}
	 Let $f\in L^q(w)$, where $1\leq  q\leq \infty$ and $w:\R_+\to \R_+$ is a weight function satisfying 
	 \begin{enumerate}[label=\textnormal{(}\roman{*}\textnormal{)}]
	 	\item $\displaystyle\sup_{x\in(0,1)}x^{\alpha+1/2}w(x)^{-1}+ \sup_{x\in (1,\infty)} x^\gamma w(x)^{-1}<\infty$ for some $\gamma>-1$, if $q=1$\textnormal{;}
	 	\item $\displaystyle\int_0^1 x^{(\alpha+1/2)q'}w(x)^{-q'/q}\, dx+\int_1^\infty x^{\gamma q'}w(x)^{-q'/q}\, dx<\infty$ for some $\gamma>-1$, if $1<q<\infty$\textnormal{;}
	 	\item $\displaystyle\int_0^1 x^{\alpha+1/2}w(x)^{-1}\, dx+\int_1^\infty x^{\gamma }w(x)^{-1}\, dx<\infty$ for some $\gamma>-1$, if $q=\infty$.
	 \end{enumerate}
	 Then the functional
	\begin{align}
	H_\alpha f:\mathcal{H}_\alpha&\to \C \nonumber \\
	\varphi&\mapsto \langle H_\alpha f,\varphi \rangle = \langle f,H_\alpha \varphi\rangle,\label{EQhankeltrFUNCTIONAL}
	\end{align}
	is continuous.
\end{proposition}
\begin{proof}
	Let $\varphi\in \mathcal{H}_\alpha$. By H\"older's inequality, we have
	\[
	|\langle H_\alpha f,\varphi \rangle|\leq \int_0^\infty |f(x)H_\alpha \varphi(x)|\, dx\leq \begin{cases} \Vert f\Vert_{L^q(w)}\Vert w^{-1/q}H_\alpha \varphi\Vert_{L^{q'}},&\text{if }1\leq q<\infty,\\
	\Vert f\Vert_{L^\infty(w)}\Vert  w^{-1}H_\alpha \varphi\Vert_{L^{1}},&\text{if }q=\infty.
	\end{cases}
	\]
	In order to estimate the weighted $L^{q'}$ norm of $H_\alpha \varphi$, we first obtain pointwise estimates for such a function. On the first place, for $x\leq 1$ one has $|H_\alpha\varphi(x)|\leq C_{\alpha,\varphi} x^{\alpha+1/2}$. Indeed, since $|J_\alpha(z)|\lesssim z^{\alpha}$ for $z< 1$ and $|J_\alpha(z)|\lesssim z^{-1/2}$ for $z\geq 1$, we have
	\begin{align*}
	|H_\alpha \varphi(x)|&\lesssim x^{\alpha+1/2}\int_0^{1/x} t^{\alpha+1/2}|\varphi(t)|\, dt+x^{1/2}\bigg|\int_{1/x}^\infty t^{1/2}\varphi(t)J_\alpha(xt)\, dt\bigg|\\
	&\lesssim x^{\alpha+1/2}\int_0^\infty t^{\alpha+1/2}|\varphi(t)|\, dt+x^{\alpha+1/2}\int_{1/x}^\infty t^{\alpha+1/2}|\varphi(t)|\, dt\lesssim  x^{\alpha+1/2}\int_0^\infty t^{\alpha+1/2}|\varphi(t)|\, dt\\
	&=x^{\alpha+1/2}\bigg(\int_0^1 t^{\alpha+1/2}|\varphi(t)|\, dt + \int_1^\infty \frac{1}{t^2}t^{\alpha+5/2}|\varphi(t)|\, dt\bigg)\\
	&\lesssim x^{\alpha+1/2}\bigg( \sup_{t\in\R_+}|\varphi(t)|+\sup_{t\in\R_+} t^{\alpha+5/2}|\varphi(t)|\bigg).
	\end{align*}
	Secondly, for $x\geq 1$ and any $\gamma>-1$ there holds $|H_\alpha \varphi(x)|\leq C'_{\alpha,\varphi} x^{\gamma}$. Indeed, integration by parts together with \eqref{EQprimitiveest} yield
	\begin{align*}
	|H_\alpha \varphi(x)|&\lesssim x^{\alpha+1/2}\int_0^{1/x} t^{\alpha+1/2}|\varphi(t)|\, dt+x^{1/2}\bigg|\int_{1/x}^\infty t^{1/2}\varphi(t)J_\alpha(xt)\, dt\bigg|\\
	&\leq x^{\gamma}\int_0^{1/x}t^{\gamma}|\varphi(t)|\, dt+x^{1/2}|K_{x}^\alpha(1/x)\varphi(1/x)|+x^{1/2}\int_{1/x}^\infty |K_{x}^\alpha(t)\varphi'(t)|\, dt\\
	&\lesssim x^{\gamma}\sup_{t\in \R_+}|\varphi(t)|+x^{-1}\sup_{t\in \R_+}|\varphi(t)|+x^{-1}\int_{1/x}^\infty |\varphi'(t)|\, dt\\
	&\lesssim  x^{\gamma}\sup_{t\in \R_+}|\varphi(t)| +x^{\gamma}\int_{1/x}^1 |\varphi'(t)|\, dt+x^{\gamma}\int_1^{\infty}\frac{1}{t^2}t^2|\varphi'(t)|\, dt\\
	&\lesssim x^{\gamma}\bigg(\sup_{t\in \R_+}|\varphi(t)|+\sup_{t\in \R_+}|\varphi'(t)|+\sup_{t\in \R_+}|t^2\varphi'(t)|\bigg).
	\end{align*}

	Assume first that $1< q< \infty$. Then
	\begin{align*}
	\Vert w^{-1/q}H_\alpha \varphi\Vert_{L^{q'}}&\asymp \bigg( \int_0^1 w(x)^{-q'/q}|H_\alpha \varphi(x)|^{q'}\, dx\bigg)^{1/q'} +\bigg(\int_1^\infty w(x)^{-q'/q}|H_\alpha  \varphi(x)|^{q'}\, dx\bigg)^{1/q'}\\
	&\lesssim \bigg(\int_0^1 x^{(\alpha+1/2)q'}w(x)^{-q'/q}\, dx \bigg)^{1/q'}\bigg( \sup_{t\in \R_+}|\varphi(t)|+\sup_{t\in \R_+} t^{\alpha+5/2}|\varphi(t)|\bigg)\\
	&\phantom{=}+\bigg(\int_{1}^\infty x^{\gamma q'}w(x)^{-q'/q}\, dx\bigg)^{1/q'}\bigg(\sup_{t\in \R_+}|\varphi(t)|+\sup_{t\in \R_+}|\varphi'(t)|+\sup_{t\in \R_+}|t^2\varphi'(t)|\bigg).
	\end{align*}
	Note that the suprema involving $\varphi$ and $\varphi'$ need not be functionals from the collection of seminorms \eqref{EQseminorms}, but they can be trivially estimated from above by linear combinations of those.
	
	For the case $q=\infty$, similar calculations yield $H_\alpha f\in \mathcal{H}'_\alpha$. Finally, if $q=1$,
	\begin{align*}
	\Vert w^{-1}H_\alpha \varphi\Vert_{L^{\infty}}&\leq \sup_{x\in (0,1)} w(x)^{-1}|H_\alpha\varphi (x) |+  \sup_{x\in (1,\infty)} w(x)^{-1}|H_\alpha \varphi(x)|\\
	&\lesssim \bigg(\sup_{x\in (0,1)}x^{\alpha+1/2}w(x)^{-1}\bigg)\bigg( \sup_{t\in\R_+}|\varphi(t)|+\sup_{t\in\R_+} t^{\alpha+5/2}|\varphi(t)|\bigg)\\
	&\phantom{=}+\bigg(\sup_{x\in (1,\infty)} x^{\gamma}w(x)^{-1}\bigg)\bigg(\sup_{t\in \R_+}|\varphi(t)|+\sup_{t\in \R_+}|\varphi'(t)|+\sup_{t\in \R_+}|t^2\varphi'(t)|\bigg),
	\end{align*}
	which completes the proof.
\end{proof}
Proposition~\ref{PROPdistrweightedlebesgue} allows to easily derive sufficient conditions on the parameters $p,q$, so that $f\in L^q_{t(p,q)}$ induces a continuous operator $H_\alpha f\in \mathcal{H}'_\alpha$.

\begin{corollary}
	Let $1\leq q\leq \infty$ and $0<p\leq\infty$. Let $f\in L^q_{t(p,q)}$. Then, $H_\alpha f\in \mathcal{H}'_\alpha$, provided that
	\begin{enumerate}[label=\textnormal{(}\roman{*}\textnormal{)}]
		\item $\dfrac{1}{\alpha+3/2}\leq p< \infty$, if $q=1$\textnormal{;}
		\item $\dfrac{1}{\alpha+3/2}<p<\infty$, if $1< q\leq \infty$.
	\end{enumerate}
\end{corollary}
\begin{proof}
	It is a direct consequence of Proposition~\ref{PROPdistrweightedlebesgue} with different choices of $w$: for $q=1$, we use   $w(x)=x^{-1/p'}$, for $1< q<\infty$ we use $w(x)=x^{q/p-1}$, and finally, for $q=\infty$ we use $w(x)=x^{1/p}$.
\end{proof}

\subsection{Lorentz spaces $L^{p,q}$}

We now show that if $f$ is a function from a certain Lorentz space it also induces continuous operator $H_\alpha f\in \mathcal{H}_\alpha'$. First, let us introduce the following notation. For $1\leq p,q\leq \infty$, we say that an integral operator $T$ is of type $(p,q)$ if $T:L^p\to L^q$ is bounded. Here we need Calder\'on's rearrangement inequality \cite{calderon} (see also \cite{heinigWeighted}).
\begin{theorem}\label{THMcalderon}
	Let $T$ be a sublinear operator of types $(1,\infty)$ and $(a,a')$, for some $1<a<\infty$. Then
	\[
	(T\varphi)^*(y)\lesssim \int_0^{1/y}\varphi^*(x)\, dx+\frac{1}{y^{a'}}\int_{1/y}^\infty \frac{\varphi^*(x)}{x^{a'}}\, dx
	\]
\end{theorem}
\begin{remark}\label{REMtype22}
	The Hankel transform \eqref{EQhankeltr} is of types $(1,\infty)$ and $(2,2)$ for every $\alpha\geq -1/2$, see \cite{DC,macaulay}.
\end{remark}
\begin{proposition}
	Let $f\in L^{p,q}$, with $1<p<\infty$ and $1\leq q\leq \infty$. Then the functional $H_\alpha f$ defined by \eqref{EQhankeltrFUNCTIONAL}	is continuous.
\end{proposition}
\begin{proof}
	Let $\varphi \in \mathcal{H}_\alpha$. By H\"older's inequality on Lorentz spaces (cf. \cite[Ch. IV, Theorem 4.7]{BSbook}) and the fact that $\Vert g\Vert_{L^{p,r}}\lesssim \Vert g\Vert_{L^{p,s}}$ for any $s\leq r$ (see \cite[Ch. I]{Grafakosbook}), we have
	\[
	|\langle H_\alpha f,\varphi \rangle|\leq \int_0^\infty |f(x)H_\alpha \varphi(x)|\, dx\leq \Vert f\Vert_{L^{p,q}}\Vert H_\alpha \varphi\Vert_{L^{p',q'}}\lesssim \Vert f\Vert_{L^{p,q}}\Vert H_\alpha \varphi\Vert_{L^{p',1}}.
	\]
	We now estimate $\Vert  H_\alpha \varphi\Vert_{L^{p',1}}$ from above by a finite linear combination of seminorms of $\varphi$ on $\mathcal{H}_\alpha$, which will yield $H_\alpha f\in \mathcal{H}_\alpha'$. We have, by Theorem~\ref{THMcalderon} (see also Remark~\ref{REMtype22}),
	\begin{align*}
	\Vert H_\alpha \varphi\Vert_{L^{p',1}}&= \int_0^\infty x^{-1/p}(H_\alpha \varphi)^*(x)\, dx\lesssim \int_0^\infty x^{-1/p}\int_0^{1/x} \varphi^*(t)\, dt\, dx \\
	&\phantom{=}+ \int_0^\infty x^{-2-1/p}\int_{1/x}^\infty \frac{\varphi^*(t)}{t^2}\, dt\, dx.
	\end{align*}
	On the one hand, since $\varphi^*$ is decreasing, $\varphi^*(0)=\sup_{x\in \R_+}|\varphi(x)|$, and $\Vert\varphi^*\Vert_1=\Vert\varphi\Vert_1$ (see \cite{BSbook,Grafakosbook}),
	\begin{align*}
	\int_0^\infty x^{-1/p}\int_0^{1/x} \varphi^*(t)\, dt\, dx &=\int_0^\infty \varphi^*(t)\, \int_0^{1/t}x^{-1/p}\, dx\, dt \asymp \int_0^\infty t^{-1/p'}\varphi^*(t)\, dt
	\\
	&= \sup_{x\in \R_+}|\varphi(x)| + \int_1^\infty \varphi^*(t)\, dt\lesssim \sup_{x\in \R_+}|\varphi(x)|  +\Vert \varphi\Vert_1 \\
	&\lesssim \sup_{x\in \R_+}|\varphi(x)|+\sup_{x\in \R_+}|x^2\varphi(x)|,
	\end{align*}
On the other hand, similarly as before,
	\begin{align*}
	\int_0^\infty x^{-2-1/p}\int_{1/x}^\infty \frac{\varphi^*(t)}{t^2}\, dt\, dx&=\int_0^\infty \frac{\varphi^*(t)}{t^2}\int_{1/t}^\infty x^{-2-1/p}\, dx\, dt\asymp \int_0^\infty t^{-1/p'}\varphi^*(t)\, dt\\
	&\leq \sup_{x\in \R_+}|\varphi(x)|+\int_1^\infty \varphi^*(t)\,dt \lesssim \sup_{x\in \R_+}|\varphi(x)| + \sup_{x\in \R_+}|x^2\varphi(x)|. 
	\end{align*}
	Combining all estimates, we get
	\[
	|\langle H_\alpha f,\varphi \rangle|\leq C_{p,q}\Vert f\Vert_{L^{p,q}} \bigg( \sup_{x\in \R_+}|\varphi(x)| + \sup_{x\in \R_+}|x^2\varphi(x)|\bigg),
	\]
	which yields the desired result.
	\end{proof}

\section{Boas' conjecture}\label{SECmainresults}

The goal of this section is to prove Theorems~\ref{THMmainlebesgue}~and~\ref{THMmainlorentz}. The approaches we follow are similar to those considered in \cite{sagherboas} and \cite{bootonlorentz}, respectively. It is worth emphasizing, as mentioned at the beginning of Section~\ref{SECdefinitenessHa}, that the inversion formula \eqref{EQinversion} holds for $GM$ functions from the weighted Lebesgue space $L^q_{t(p,q)}$ with $1\leq q\leq \infty$ and $\dfrac{1}{\alpha+3/2}<p<\infty$ (and thus, also for those from the Lorentz space $L^{p,q}$ with $1\leq q\leq \infty$ and $1<p<\infty$, by Theorem~\ref{THMbooton}).

\subsection{Weighted Lebesgue norm inequalities}

First of all, we prove a Pitt-type inequality for the Hankel transform of $GM$ functions that  includes  the cases $q=1,\infty$ (for the case $1<q<\infty$ this was proved in \cite{DCGT,miowni}).

\begin{theorem}\label{THMpitt1infty}
	Let $f\in GM$, $1\leq q\leq \infty$, and $\dfrac{1}{\alpha+3/2}<p<\infty$. If $f\in L^q_{t(p,q)}$, then $H_\alpha f\in  L^q_{t(p',q)}$ and
	\[
	\| H_\alpha f\|_{L^q_{t(p',q)}}\lesssim \| f\|_{L^q_{t(p,q)}}.
	\]
\end{theorem}
In order to prove Theorem~\ref{THMpitt1infty} we will need Hardy's inequalities \cite[p. 20]{zygmund}.
\begin{thmletter}
	Let $1\leq q<\infty$ and $\sigma>0$. Then, for every measurable $f$,
	\[
	\int_0^\infty \bigg( y^{-\sigma} \int_0^y |f(x)|\, \frac{dx}{x} \bigg)^q \frac{dy}{y} \lesssim \int_0^\infty \big( x^{-\sigma} |f(x)|\big)^q\, \frac{dx}{x},
	\]
	and
	\[
	\int_0^\infty \bigg( y^{\sigma} \int_y^\infty  |f(x)|\, \frac{dx}{x} \bigg)^q \frac{dy}{y} \lesssim  \int_0^\infty \big( x^{\sigma} |f(x)|\big)^q\, \frac{dx}{x},
	\]
	where the involved constants do not depend on $f$.
\end{thmletter}
\begin{proof}[Proof of Theorem~\ref{THMpitt1infty}]
	We proceed similarly as in Theorem 4 of \cite{bootonlorentz}, where an analogous result was proved for sine and cosine transforms. First of all, it follows by Lemma~\ref{LEMimproperdef} that $H_\alpha f$ is well defined as an improper integral. We now apply the estimate \eqref{EQbesselestimate} to obtain, for any $t>0$,
	\begin{align*}
	|H_\alpha f(y)|\lesssim y^{\alpha+1/2}\int_{0}^{t}  x^{\alpha+1/2}|f(x)|\, dx+y^{1/2}\bigg|\int_{t}^\infty x^{1/2}f(x)J_\alpha(xy)\, dx\bigg|.
	\end{align*}
	Integration by parts, the estimate \eqref{EQprimitiveest}, and  the fact that $f$ vanishes at infinity (which follows from $f\in L^{q}_{t(p,q)}$ and $\dfrac{1}{\alpha+3/2}<p<\infty$, by  Lemma~\ref{LEMgmproperties}) imply that
	\begin{align*}
	y^{1/2}\bigg|\int_{t}^\infty x^{1/2}f(x)J_\alpha(xy)\, dx\bigg|\lesssim \frac{1}{y}|f(t)|+\frac{1}{y}\int_{t}^\infty |df(x)| \lesssim \frac{1}{y}\int_{t}^\infty |df(x)|,
	\end{align*}
	where in the last inequality we used the estimate $|f(t)|\leq \int_t^\infty |df(x)|$, which is valid since $f$ vanishes at infinity.	Thus, we deduce by (iii) of Lemma~\ref{LEMgmproperties},
	\[
	|H_\alpha f(y)|\lesssim  y^{\alpha+1/2}\int_{0}^{t}  x^{\alpha+1/2}|f(x)|\, dx+\frac{1}{y}\int_{t/\lambda}^\infty \frac{|f(x)|}{x}\, dx.
	\]
	Note that since $f\in L^{q}_{t(p,q)}$, $1\leq q\leq \infty$, and $\dfrac{1}{\alpha+3/2}<p<\infty$, the right-hand side is finite, by H\"older's inequality. Hence, by letting $t=1/y$ we obtain
	\begin{align*}
	\| H_\alpha f\|_{L^q_{t(p',q)}} &=\bigg( \int_0^\infty \Big( y^{1/p'} |H_\alpha f(y)|\Big)^q \frac{dy}{y}\bigg)^{1/q} \\
	&\lesssim \bigg(\int_0^\infty   \bigg( y^{\alpha+1/2+1/p'}  \int_0^{1/y}  x^{\alpha+1/2}|f(x)|\, dx  \bigg)^q \frac{dy}{y} \bigg)^{1/q}\\
	& \phantom{=}+ \bigg(\int_0^\infty   \bigg(  y^{-1/p}\int_{1/(\lambda y)}^\infty|f(x)| \frac{dx}{x}\bigg)^q \frac{dy}{y} \bigg)^{1/q}\\
	&=  \bigg(\int_0^\infty   \bigg( y^{-\alpha-1/2-1/p'}  \int_0^{y}  x^{\alpha+3/2}|f(x)|\, \frac{dx}{x}  \bigg)^q \frac{dy}{y} \bigg)^{1/q}\\
	&\phantom{=}+\bigg(\int_0^\infty   \bigg(  y^{1/p}\int_{y/\lambda}^\infty|f(x)| \frac{dx}{x}\bigg)^q \frac{dy}{y} \bigg)^{1/q},
	\end{align*}
	where in the last inequality we applied the change of variables $y\to 1/y$. On the one hand, since $p>\dfrac{1}{\alpha+3/2}$, Hardy's inequality yields
	\[
	\bigg(\int_0^\infty   \bigg( y^{-\alpha-1/2-1/p'}  \int_0^{y}  x^{\alpha+3/2}|f(x)|\, \frac{dx}{x}  \bigg)^q \frac{dy}{y} \bigg)^{1/q} \lesssim \bigg(\int_0^\infty \big( x^{1/p}|f(x)|\big)^q \frac{dx}{x}\bigg)^{1/q}=\| f\|_{L^q_{t(p,q)}},
	\]
	whilst on the other hand, again by Hardy's inequality,
	\[
	\bigg(\int_0^\infty   \bigg(  y^{1/p}\int_{y/\lambda}^\infty|f(x)| \frac{dx}{x}\bigg)^q \frac{dy}{y} \bigg)^{1/q} \lesssim  \bigg(\int_0^\infty \big( x^{1/p}|f(x)|\big)^q \frac{dx}{x}\bigg)^{1/q}=\| f\|_{L^q_{t(p,q)}}.
	\]
	The case $q=\infty$ is similar and is omitted (in fact, this complementary case is dealt with in full detail in the case of Lorentz spaces, in Theorem~\ref{THMlorentz} below; note that Hardy's inequalities are not needed in this case).
\end{proof}

\begin{lemma}\label{LEMbddnesslebesgue}
	Let $f\in L^q_{t(p,q)}$, with
	\begin{enumerate}[label=\textnormal{(}\roman{*}\textnormal{)}]
	\item $q=1$ and $\dfrac{1}{\alpha+3/2}\leq p<\infty$, or
	\item $1<q\leq \infty$ and $\dfrac{1}{\alpha+3/2}<p<\infty$.
\end{enumerate}
Then the inequality
	\begin{equation}
	\label{EQmaximallebesgue}
	 \Vert M\Phi_{H_\alpha f} \Vert_{L^{p',q}} =\Vert M\Phi_{H_\alpha f}\Vert_{L^q_{t(p',q)}}\lesssim  \Vert f\Vert_{L^q_{t(p,q)}}.
	\end{equation}
holds for any $\varphi\in \mathcal{H}_\alpha$.
\end{lemma}
\begin{remark}
	Given $\varphi:\R_+\to \C$, the operator $\Phi_g$ was defined in \eqref{EQdefauxiliary} for a given function $g$. However, if $\varphi\in \mathcal{H}_\alpha$ and $f$ is a function for which $H_\alpha f\in \mathcal{H}'_\alpha$, abusing of notation we may write
	\[
	\Phi_{H_\alpha f}(t)=\langle f,H_\alpha\varphi_t\rangle,
	\]
	as done in \eqref{EQmaximallebesgue}, taking into account the definition of $H_\alpha f$	\eqref{EQhankeldefdistribution}. This notation is adopted in what follows.
\end{remark}
\begin{proof}[Proof of Lemma~\ref{LEMbddnesslebesgue}]
	The proof is carried out exactly in the same lines as \cite[Theorem 3.1]{sagherboas}. Indeed, H\"older's inequality implies
	\[
	|\Phi_{H_\alpha f}(t)|\leq t^{-1/p'}\Vert f\Vert_{L^q_{t(p,q)}}\Vert H_\alpha\varphi \Vert_{L^{q'}_{t(p',q')}},
	\]
	so that the operator $T:f\mapsto M\Phi_{H_\alpha f}$ maps $L^q_{t(p,q)}$ into $L^{p',\infty}$. Fixing $q$ and interpolating between different values of $p$, the interpolation theorem with change of measures by Stein and Weiss (cf. \cite{SWinterpolation}) yields
	\[
	T:L^{q}_{t(p,q)}\to L^{p',q},
	\]
	as desired.
\end{proof}

Finally, we are in a position to prove our main result concerning weighted Lebesgue spaces.
\begin{proof}[Proof of Theorem~\ref{THMmainlebesgue}]
	It follows from Theorem~\ref{THMpitt1infty}  that
	\[
	\Vert H_\alpha f\Vert_{L^{q}_{t(p',q)}}\lesssim \Vert f\Vert_{L^q_{t(p,q)}}, \qquad 1\leq q\leq \infty, \qquad \frac{1}{\alpha+3/2}<p<\infty.
	\]
By Lemma~\ref{LEMbddnesslebesgue} (with $H_\alpha f$ in place of $f$), we get
\[
\Vert M\Phi_{f} \Vert_{L^{q}_{t(p,q)}}\lesssim \Vert H_\alpha f\Vert_{L^{q}_{t(p',q)}},
\]
for any $\varphi\in \mathcal{H}_\alpha$. Finally, Theorem~\ref{THMaverageoperator} together with the appropriate choice of $\varphi$ yields
\[
\Vert f\Vert_{L^q_{t(p,q)}}\lesssim \Vert M\Phi_{f} \Vert_{L^{q}_{t(p,q)}},
\]
with all the estimates valid for the ranges $1\leq q\leq \infty$ and $\dfrac{1}{\alpha+3/2}<p<\infty$. The hypothesis $x^rf(x)\to 0$ as $x\to 0$ needed to apply Theorem~\ref{THMaverageoperator} follows from the fact that $f\in L^q_{t(p,q)}$ and Lemma~\ref{LEMgmproperties}.
\end{proof}
\begin{remark}
	Note that in Theorem~\ref{THMaverageoperator}, rather than choosing $\varphi=\chi_{(0,1)}$, we allow $\varphi$ to be supported on $(0,1+\varepsilon/2)$, so that it is also valid for some choice of $\varphi\in \mathcal{H}_\alpha$, which is needed to prove Theorem~\ref{THMmainlebesgue}.
\end{remark}
\subsection{Lorentz norm inequalities}
In order to prove Theorem~\ref{THMmainlorentz}, we need to establish some auxiliary estimates on Lorentz norms.

\begin{theorem}\label{THMlorentz}
	Let $f\in GM$, and assume that $f\in L^{p,q}$ with $1<p<\infty$ and $1\leq q\leq \infty$. Then $H_\alpha f\in L^{p',q}$, and moreover
	\[
	\Vert H_\alpha f\Vert_{L^{p',q}}\lesssim \Vert f\Vert_{L^{p,q}}.
	\]
\end{theorem}
\begin{proof}
	First of all, we apply the estimate \eqref{EQbesselestimate} to obtain, for any $t>0$,
	\begin{align*}
	|H_\alpha f(y)|\lesssim \int_{0}^{t} |f(x)|\, dx+y^{1/2}\bigg|\int_{t}^\infty x^{1/2}f(x)J_\alpha(xy)\, dx\bigg|.
	\end{align*}
	Integration by parts, the estimate \eqref{EQprimitiveest}, and  the fact that $f$ vanishes at infinity (which follows from $f\in L^{q}_{t(p,q)}$ (cf. Theorem~\ref{THMbooton}) and Lemma~\ref{LEMgmproperties}) imply that
	\begin{align*}
	y^{1/2}\bigg|\int_{t}^\infty x^{1/2}f(x)J_\alpha(xy)\, dx\bigg|\lesssim \frac{1}{y}|f(t)|+\frac{1}{y}\int_{t/\lambda}^\infty \frac{|f(x)|}{x}\, dx,
	\end{align*}
	and thus we deduce
	\[
	(H_\alpha f)^*(y)\lesssim \int_{0}^{t} |f(x)|\, dx+\frac{1}{y}|f(t)| +\frac{1}{y}\int_{t/\lambda}^\infty \frac{|f(x)|}{x}\, dx\lesssim \int_{0}^{t} |f(x)|\, dx+\frac{1}{y}\int_{t/\lambda}^\infty \frac{|f(x)|}{x}\, dx,
	\]
	where the right-hand side is finite, since $f\in GM$ and $f\in L^{q}_{t(p,q)}$. From this point, the proof for the case $1\leq q<\infty$ is exactly the same as the one of \cite[Theorem 4]{bootonlorentz} (and similar to that of Theorem~\ref{THMpitt1infty}) and is therefore omitted. We give a detailed proof for the case $q=\infty$. Since $y^{1/p'}\asymp y^{1+1/(2p')}\int_0^{1/y} t^{-1/(2p')}\, dt$ for $y>0$, we have
	\begin{align*}
	y^{1/p'} (H_\alpha f)^*(y)&\asymp y^{1+1/(2p')}\int_0^{1/y}t^{-1/(2p')}(H_\alpha f)^*(y) \, dt \lesssim y^{1+1/(2p')}\int_0^{1/y}t^{-1/(2p')}\bigg(\int_0^t|f(x)|\, dx\bigg)\, dt\\
	&\phantom{=}+y^{1/(2p')}\int_{0}^{1/y}t^{-1/(2p')} \bigg(\int_{t/\lambda}^{\infty} \frac{|f(x)|}{x}\, dx\bigg)\, dt\\
	&\lesssim \Vert f\Vert_{L^\infty_{t(p,\infty)}}\bigg( y^{1+1/(2p')}\int_0^{1/y}t^{-1/(2p')}\bigg(\int_0^t x^{-1/p}\, dx \bigg)\, dt\\
	&\phantom{=}+y^{1/(2p')}\int_0^{1/y} t^{-1/(2p')}\bigg(\int_{t/\lambda}^{\infty}x^{-1-1/p}\,dx\bigg)\, dt\bigg)\\
	&\asymp \Vert f\Vert_{L^\infty_{t(p,\infty)}},
	\end{align*}
	i.e., $\Vert H_\alpha f\Vert_{L^{p',\infty}}\lesssim \Vert f\Vert_{L^\infty_{t(p,\infty)}}\asymp \Vert f\Vert_{L^{p,\infty}}$.
\end{proof}

We now prove a relation between the norm of $f$ from a certain Lorentz space and the corresponding norm of $M\Phi_{H_\alpha f}$ in the corresponding space (cf. \eqref{EQdefauxiliary} and \eqref{EQdefmaximal}), given $\varphi\in \mathcal{H}_\alpha$. This is an extension of the result by Y. Sagher for the Fourier transform given in \cite{sagherboas} and is proved in the same way.

\begin{lemma}\label{LEMaverageoperator}
	Let $1<p<\infty$ and $0< q\leq \infty$. If $f\in L^{p,q}$, then
	\begin{equation}
	\label{EQmaximal}
	\Vert M\Phi_{H_\alpha f}\Vert_{L^{p',q}}\leq C_{\varphi,p}\Vert f\Vert_{L^{p,q}}.
	\end{equation}
\end{lemma}
\begin{proof}
	Let $f\in L^p$. Since $H_\alpha \varphi_t(u)=H_\alpha \varphi(tu)$, we have, by H\"older's inequality,
	\begin{align*}
	|\Phi_{H_\alpha f}(t)|=|\langle \varphi_t,H_\alpha f\rangle|\leq \Vert f\Vert_{L^p} \bigg(\int_0^\infty |H_\alpha \varphi(ut)|^{p'}\, du\bigg)^{1/p'} = t^{-1/p'}\Vert f\Vert_{L^p}\Vert H_\alpha \varphi\Vert_{L^{p'}}.
	\end{align*}
	Hence, $t^{1/p'}M\Phi_{H_\alpha f}(t)\leq \Vert f\Vert_{L^p}\Vert H_\alpha \varphi\Vert_{L^{p'}}$. In other words, the sublinear operator $T$ defined by $Tf=M\Phi_{H_\alpha f}$ is bounded from $L^p=L^{p,p}$ to $L^{p',\infty}$. Interpolating, we obtain the boundedness of the operator $T$ from $L^{p,q}$ to $L^{p',q}$ for any $0< q\leq \infty$ (see \cite[Theorem 26]{sagherinterpo}), i.e., \eqref{EQmaximal} holds.	
\end{proof}
\begin{corollary}
	Let $1<p<\infty$ and $1\leq q\leq \infty$. If $H_\alpha f\in L^{p',q}$ and $f$ is a $GM$ function, then $f\in L^{p,q}$.
\end{corollary}
\begin{proof}
	By Lemma~\ref{LEMaverageoperator} and Theorem~\ref{THMaverageoperator}, we get
	\[
	\Vert f\Vert_{L^{q}_{t(p,q)}}\lesssim \Vert H_\alpha f\Vert_{L^{p',q}}.
	\]
	Finally, Theorem~\ref{THMbooton} yields the desired result.
\end{proof}
We are now in a position to prove Theorem~\ref{THMmainlorentz}.
\begin{proof}[Proof of Theorem~\ref{THMmainlorentz}]
First of all, combining Theorem~\ref{THMlorentz} and Lemma~\ref{LEMaverageoperator} we obtain
\[
\Vert M\Phi_f\Vert_{L^q_{t(p,q)}}=\Vert M\Phi_f \Vert_{L^{p,q}}\lesssim \Vert H_\alpha f\Vert_{L^{p',q}}\lesssim \Vert f\Vert_{L^{p,q}},
\]
for $\varphi \in \mathcal{H}_\alpha$. Now, Theorem~\ref{THMaverageoperator} together with the appropriate choice of $\varphi$ yields $\Vert f\Vert_{L^q_{t(p,q)}}\lesssim \Vert M\Phi_{f}\Vert_{L^q_{t(p,q)}}$. Finally, Theorem~\ref{THMbooton} completes the proof.
\end{proof}

Putting together Theorems~\ref{THMmainlebesgue}, \ref{THMmainlorentz}, and \ref{THMbooton}, we can derive the following equivalence.
\begin{corollary}\label{CORhankel-transforms}
	Let $f\in GM$ be real-valued and let $1<p,q<\infty$. Then, for any $\alpha\geq -1/2$,
	\[
	f\in L^q_{t(p,q)}\Leftrightarrow H_\alpha f\in L^q_{t(p',q)}\Leftrightarrow f\in L^{p,q}\Leftrightarrow H_\alpha f\in L^{p',q}.
	\]
\end{corollary}
\subsection{Boas' conjecture for the Fourier transform}

\subsubsection{One-dimensional Fourier transforms}Let $f$ be a function defined on $\R$. We denote 
\[
f_e(x)=\frac{f(x)+f(-x)}{2},\qquad f_o(x)=\frac{f(x)-f(-x)}{2},
\]
the even and odd part of $f$, respectively, so that $f=f_e+f_o$. Theorems~\ref{THMmainlebesgue} and~\ref{THMmainlorentz} together with \eqref{EQbesselfns1/2} and the well-known representation of the Fourier transform $\widehat{f}=H_{-1/2}f_e+iH_{1/2}f_o$ allow us to easily derive the solution to the Boas' conjecture for the Fourier transform, in the case of real-valued $GM$ functions. For the sake of completeness, we first prove a preliminary lemma.
\begin{lemma}\label{LEMMA-lebesgue-decomp}
	Let $f:\R\to \C$. Let $w:\R\to \R_+$ be an even weight and $0< q\leq \infty$. Then $f\in L^q_\R(w)$ if and only if $f_e,f_o\in L^q_\R(w)$.
\end{lemma}
\begin{proof}
	The ``if'' part is trivial. For the ``only if'' part, we have, in the case $q<\infty$,
	\begin{align*}
	\int_{\R} w(x)|f(x)|^q\, dx=\int_{0}^\infty w(x) \big(|f_e(x)+f_o(x)|^q+|f_e(x)-f_o(x)|^q\, dx \big)\geq \int_0^\infty w(x)|f_e(x)|^q\, dx,
	\end{align*}
	where we used the inequality $|a+b|^q\leq 2^{q}(|a|^q+|b|^q)$. This shows that $f_e\in L^q_\R(w)$ and therefore also $f_o=f-f_e\in L^q_\R(w)$. For the case $q=\infty$, triangle inequality yields
	\[
	\sup_{x\in \R} w(x)|f(x)|\geq \frac{1}{2}\sup_{x\in (0,\infty)} w(x)|f_e(x)+f_o(x)|+\frac{1}{2}\sup_{x\in (0,\infty)}w(x)|f_e(x)-f_o(x)|\geq \sup_{x\in (0,\infty)} w(x)|f_e(x)|,
	\]
	and the result follows similarly as before.
\end{proof}
\begin{lemma}\label{LEMMA-lorentz-decomp}
	Let $f:\R\to \C$ and $0<p,q\leq \infty$. Then $f\in L^{p,q}_\R$ if and only if $f_e,f_o\in L^{p,q}_\R$.
\end{lemma}
\begin{proof}
	Again, the ``if'' part is trivial. For the ``only if'' part, note that
	\[
	d_f(s)=\frac{1}{2}\big( |\{x\in \R:|f_e(x)+f_o(x)|>s\}|+|\{x\in \R:|f_e(x)-f_o(x)|>s\}|   \big).
	\]
	Since 
	\[
	\frac{1}{2} |\{x\in \R:|f_e(x)|>s\}|\leq |\{x\in \R:|f_e(x)+f_o(x)|>s\}|+|\{x\in \R:|f_e(x)-f_o(x)|>s\}|,
	\]
	or in other words, $d_{f_e}(s)\leq 4d_f(s)$, it follows that $f_e\in L^{p,q}_\R$ by \eqref{EQlorentz-norm-alternative}, and also $f_o=f-f_e\in L^{p,q}_\R$.
\end{proof}

We are in a position to prove Corollary~\ref{COR-fourier}, dealing with one-dimensional Fourier transforms.
\begin{proof}[Proof of Corollary~\ref{COR-fourier}]
	The result readily follows from the representation $\widehat{f}=H_{-1/2}f_e+iH_{1/2}f_o$, together with Corollary~\ref{CORhankel-transforms} and Lemmas~\ref{LEMMA-lebesgue-decomp}~and~\ref{LEMMA-lorentz-decomp}.
\end{proof}
The interval for $p$ in Corollary~\ref{COR-fourier} cannot be extended even for weighted Lebesgue spaces as done in Theorem~\ref{THMmainlebesgue}, where $\dfrac{1}{\alpha+3/2}<p<\infty$, because the even part of $\widehat{f}$ corresponds to the cosine transform, i.e., the Hankel transform of order $\alpha=-1/2$, and the optimal interval for the cosine transform is $1<p<\infty$, according to Theorem~\ref{THMmainlebesgue}.
\begin{proof}[Proof of Corollary~\ref{COR-fourier-radial}]
	The result follows by using the relation \eqref{EQ-fourier-radial-functions}, and by Theorem~\ref{THMmainlebesgue} with $\alpha=\dfrac{n}{2}-1$, and $\dfrac{1}{p}=\gamma-\dfrac{n-1}{2}+\dfrac{n}{q}$.
\end{proof}

	\end{document}